\documentclass[12pt]{amsart}
\usepackage{graphicx} 
\DeclareGraphicsExtensions{.pdf,.eps,.fig,.pdf_tex}
\usepackage{color}
\usepackage{amsmath}
\usepackage{amssymb}
\usepackage{amsthm}
\usepackage{amscd}
\usepackage{amsfonts}
 \usepackage{hyperref}
 \usepackage{float}

\DeclareGraphicsExtensions{.pdf,.eps,.fig}

\theoremstyle{plain} \numberwithin{equation}{section}
\newtheorem{theorem}{Theorem}[section]

\newtheorem{lemma}[theorem]{Lemma}
\newtheorem{proposition}[theorem]{Proposition}
\theoremstyle{definition}
\newtheorem{definition}[theorem]{Definition}

\newtheorem{remark}[theorem]{Remark}
\newtheorem{example}[theorem]{Example}
\title{Split canonical relations}
\author{Alberto S. Cattaneo and Ivan Contreras}
\address{Institut f\"ur Mathematik, Universit\"at Z\"urich Irchel, Winterthurerstrasse 190, CH-8057 Z\"urich, Switzerland}
\address{Department of Mathematics, Amherst College, Amherst, Massachusetts  01002, USA}

\email{cattaneo@math.uzh.ch, icontreraspalacios@amherst.edu}

\begin{document}

\begin{abstract}
A Lagrangian subspace $L$ of a weak symplectic vector space is called \emph{split Lagrangian} if it has an isotropic (hence Lagrangian) complement.  When the symplectic structure is strong, it is sufficient for $L$ to have a closed complement, which can then be moved to become isotropic.
The purpose of this note is to develop the theory of compositions and reductions of split canonical relations for symplectic vector spaces. We give conditions on a coisotropic subspace $C$ of a weak symplectic space $V$  which imply that the induced canonical relation $L_C$ from $V$ to $C/C^{\omega}$ is split, and, from these, we find sufficient  conditions  for split canonical relations to compose well. We prove that the canonical relations arising in the Poisson sigma model from the Lagrangian field theoretical approach (\cite {Relational, Lagrangian}) are split, giving a description of symplectic groupoids integrating Poisson manifolds in terms of split canonical relations.
\end{abstract}

\maketitle

\tableofcontents

\section{Introduction}
The functorial description of  classical and quantum field theories requires enhancement of the usual symplectic category, allowing infinite dimensional objects (weak symplectic manifolds), and allowing general canonical relations, not just symplectomorphisms, as morphisms. As  is explained in \cite{Weinstein}, the construction of such an extended symplectic category encounters several obstacles,  even in finite dimensions, and many alternatives have been proposed over time. These constructions face even more technical difficulties when infinite dimensional spaces are taken into account.

One of the problems arises from  the nature of the objects of this extended category: the manifolds which model the space of fields for classical field theories are in general Fr\'echet manifolds, and conventional tools such as the inverse function theorem, Frobenius theorem, etc.~fail to hold in general in the Fr\'echet category.

Another inconvenience, which is a crucial impediment to establishing functorial quantization, is the fact that the composition of canonical relations is not well defined in general. In the finite dimensional case they compose well when certain transversality conditions are satisfied; in infinite dimensions, we can find examples of canonical relations whose transversal composition is isotropic but not coisotropic (see e.g. \cite{Con}).

The objective of this paper is to analyze the special case in which we allow the objects to be (infinite dimensional) weak symplectic vector spaces and the morphisms to be \emph{split canonical relations}, which are isotropic subspaces with isotropic complements. This notion of Lagrangian subspace was first considered  in the context of infinite dimensional symplectic linear spaces in \cite{Weinstein2}, and it happens to coincide with the conventional notion (maximal isotropic) in the finite dimensional case (see e.g. \cite{Arnold, Weinstein2}). For the infinite dimensional case, we  determine sufficient conditions for split canonical relations to compose well and to have well defined reductions.

Split Lagrangian subspaces appear naturally in the Hamiltonian formulation for Lagrangian field theories with boundary, following the BV-BFV formalism, as in \cite{Lagrangian}. This type of field theory naturally produces an isotropic submanifold $L$ of the symplectic space of boundary  fields $\mathcal F_{\partial}$.
A \emph{good} choice of a boundary conditions for this theory consists of the choice of a submanifold  $L^{'}$ on which the Noether 1-form vanishes and which intersects $L$ in isolated points (the solutions to the Euler-Lagrange system of equations).
\newline
For this choice of boundary conditions, we want  to insure that the variational problem has no boundary contributions.
This condition implies first of all that $L^{'}$ must be isotropic.  Globally, this condition is even stronger, as it requires the vanishing of the class of the Noether 1-form (changing it  by exact terms is allowed as it corresponds to change the action functional by a boundary term). 
Locally, at a point of $L\cap L^{'}$, the tangent spaces to $L$ and to $L^{'}$ must then be complementary isotropic subspaces.

The last part of the paper (Section 6) is devoted to proving that the evolution relations arising from the Poisson sigma model (PSM), a two dimensional topological theory, are split canonical relations obeying a \emph{neat intersection} condition which allows compositions and reductions. These considerations help us to provide an alternative and shorter proof that such evolution relations are Lagrangian for the PSM case \cite{Relational, Con} and that the reduced version of these relations determines a \emph{split} version of the conventional symplectic groupoids for integrable Poisson manifolds.
We conjecture that the split condition for canonical relations is satisfied in a larger class of classical field theories with boundary. In particular, we intend to  study the case in which the symplectic space of fields happens to be reflexive. 

The study of split Lagrangian spaces can be naturally extended to the framework of Banach manifolds, for which the existence of split Lagrangian submanifolds implies the existence of complementary isotropic smooth distributions, which play an important role in the symplectic formulation of field theories with boundary.
\subsection*{Acknowledgments}
The authors thank Alan Weinstein for his useful insights at different stages of this project, in particular for the idea of the proof of Theorem 2.10 and the finite dimensional case (Section 3.2). A. S. C. acknowledges partial support of SNF Grant No. 200020\_172498/1. This research was (partly) supported by the NCCR SwissMAP, funded by the Swiss National Science Foundation, and by the COST Action MP1405 QSPACE, supported by COST (European Cooperation in Science and Technology). I.C. was partially supported by SNF Grant P300P2-154552. 
\section{Definitions and basic properties}

In this article, we consider topological vector spaces over $\mathbb R$.
Thus, the dual space $V^*$ of $V$ is the space of \emph{continuous} linear functionals of $V$.
The topology for Banach and Hilbert spaces is the metric topology, and the topology for Fr\'echet  spaces is the initial topology, i.e. the coarsest topology for which all the seminorms are continuous\footnote{With this topology, the Fr\'echet spaces are particular instances of locally convex spaces.}.

\begin{definition} A closed subspace $W$ of a vector space $V$ is called \emph{split} if it has a closed complement $K$ in $V$.
\end{definition}
\begin{definition} A (possibly infinite dimensional) vector space $V$ is called \emph{weak symplectic} if it is equipped with a skew symmetric bilinear form $\omega$ such that the induced linear map $\omega^{\sharp}\colon V \to V^{*}$ is injective. It is \emph{symplectic} if $\omega^{\sharp}$ is an isomorphism.
\end{definition}
Note that if the vector space $V$ is symplectic, then $(V^*)^*\cong V^*\cong V$, thus $V$ is reflexive. This does not necessarily hold for weak symplectic spaces (Proposition \ref{properties}).

$W^{\omega}$ denotes the symplectic orthogonal subspace of a subspace $W$ in a weak symplectic vector space.  It is necessarily closed, since it is determined by the condition $\omega(W, \bullet )=0$, and $\omega$ is continuous.
We leave to the reader the proof of the following useful properties of symplectic orthogonality.

\begin{proposition} \label{properties}\emph{ 
Let $V$ be a weak symplectic space and $W$, $Z$ subspaces of $V$. Then:
\begin{enumerate}
 \item $W \subseteq Z \Longrightarrow Z^{\omega} \subset W^{\omega}$.
\item $(W+Z)^{\omega} = W^{\omega} \cap Z^{\omega}$.
\item $(W \cap Z)^{\omega} \supseteq W^{\omega}+ Z^{\omega} $.
\item $W \subseteq W^{\omega \omega}$.
\item $W^{\omega}= W^{\omega \omega \omega}$.
\end{enumerate}
}
\end{proposition}

In the symplectic (e.g.~finite dimensional weak symplectic) case, we have a slightly stronger result.

\begin{proposition} \label{propertiessymplectic}\emph{ 
When $V$ is symplectic and $W$ is a closed subspace, the inequality (4) above becomes
an equality.}
\end{proposition} 

\proof{Suppose that the inequality $(4)$ is   strict.   Then there is an element $v$ of $V\backslash W$ for which the linear functional $\omega(v,\cdot)$ annihilates $W^\omega$.  On the other hand, since $W$ is closed, by the Hahn-Banach theorem (which holds in any Fr\'echet space \cite{Rudin}), there is a linear functional which takes the value 1 on $v$ and annihilates $W$.   
Since $\omega$ is symplectic, this linear functional is realized by some $z\in V$; i.e. $z\in  W^\omega$ and 
$\omega(z,v)$ = 1.  But this contradicts the assumption that $\omega(v,\cdot)$ annihilates $W^\omega$.\\
\qed}

\definition{A subspace $L$ of a weak symplectic vector space $V$ is called \emph{isotropic} if $L\subseteq L^{\omega}$, \emph{coisotropic} if $L^{\omega}\subseteq L$ and \emph{Lagrangian} if $L=L^{\omega}$.}

\begin{remark} Note that a subspace $L$ is Lagrangian if and only if it is maximal isotropic, and there are many of these subspaces, by Zorn's lemma.  A Lagrangian subspace is closed because it is the symplectic orthogonal space of itself.
\end{remark}

\proposition{\label{Isosplitting}
If $V = L \oplus L'$, where $L$ and $L'$ are isotropic, then $L$ and $L'$ are each Lagrangian.  (In particular, they are closed.)}   

\proof{If $L$ were properly contained in an isotropic subspace $L_1$, then $L_1$ would have a nontrivial intersection with $L'$, and any element of this intersection would be symplectically orthogonal to both $L$ and $L'$, in contradiction to the injectivity of $\omega^\sharp.$  This shows that $L$ is maximal isotropic, i.e. Lagrangian.   Of course, the same argument applies to $L'$.  \qed}\\

\begin{remark}We will call a subspace $L$ \emph{split Lagrangian} if it is Lagrangian and has a Lagrangian complement.  For 
example, the spaces $L$ and $L'$ in Proposition \ref{Isosplitting} are both split Lagrangian. We will also refer to $(L,L')$ as a 
splitting l-pair.
\end{remark}

Theorem \ref{Split} below shows that, when the ambient space is symplectic, and not just weakly symplectic,  any closed complement to a Lagrangian subspace can be moved to become Lagrangian itself.  Thus, in the symplectic case, 
the two possible interpretations of the term ``split Lagrangian subspace" are equivalent; i.e. a subspace which is split and Lagrangian is split Lagrangian.  We do not know whether this remains true in the general weakly symplectic case.  


\lemma{\label{Splitting} If $W\oplus Z$ is a splitting of a symplectic space $V$ by closed subspaces, then $W^{\omega} \oplus Z^{\omega}$ is also a splitting of $V$ by closed subspaces}.
\proof{The splitting $V = W \oplus Z$ leads to a splitting $V^* = W^0 \oplus Z^0$ of $V^*$ into the sum of the annihilators of $W$ and $Z$ respectively.   The inverse of the isomorphism $\omega^\sharp$ then takes $W^0$ to $W^\omega$ and $Z^0$ to $Z^\omega$, giving the required splitting of $V$.
\qed 
}\\

\theorem{\label{Split} If $V$ is symplectic, any Lagrangian subspace $L\subseteq V$ which is split as a subspace has a Lagrangian complement $L^c$ in $V$ and is hence $(L, L^c)$ is a splitting l-pair.}
\proof{ Let $K$ be a closed complement to $L$ in $V$. By Lemma \ref{Splitting},
$K^\omega$ is also a closed complement to $L$.

We define the \emph{average} of $K$ and $K^{\omega}$ to be
$$L^c:= \{\frac 1 2 (k + k^{'}): \, k \in K, \, k^{'} \in K^{\omega}, 
k-k^{'}\in L\}.$$
 By Proposition \ref{Isosplitting}, we just have to prove that $L^c$ is an isotropic complement to $L$ in $V$.
\begin{figure}[ht]
    \centering
    \def\svgwidth{200 pt}
    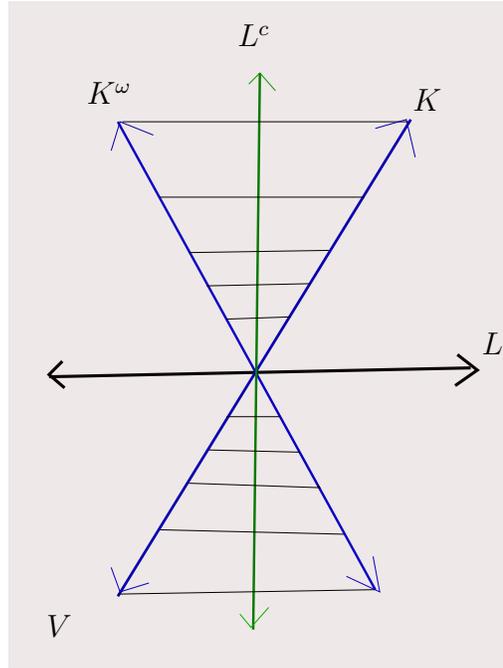
    \caption{The average construction for a split weak symplectic vector space $V$.}
    \label{fig:Corner}
\end{figure}

We first prove that it is a complement.   If $v$ in $L^c$ belongs to $L$ as well, then we must have $k+k'$ in the definition above both belonging to $L$.   But $k-k'$ belongs to $L$ as well, so $k$ and $k'$ belong to $L$.    But, since $K$ and $K'$ are complements to $L$, $k$ and $k'$ must be zero.  Thus, $L^c \cap L= \{ 0\}.$
On other hand, any element $v$ of $V$ may be expressed as either $l+k$ or $l' + k'$, where $l$ and $l'$ belong to $L$, $k \in K$, and $k' \in K'$.  Then $k - k' = l' -l$ belongs to $L$, and 
so $l_c = \frac 1 2  k + \frac 1 2 k'$ belongs to $L^c$.   Now $v$ is the average $\frac 1 2 (l + k + l' +k')$ of its two expressions, which is the sum of $l_c$ and an element of $L$.  Thus $L^c \oplus L = V.$

We will now verify that $L^c$ is isotropic. Let $l_1^c:=\frac 1 2 k_1+ \frac 1 2 k_1^{'}$ and $l_2^c:=\frac 1 2 k_2+ \frac 1 2 k_2^{'} $ be elements in $L^c$.
Then,
\begin{eqnarray*}
\omega(l_1^c, l_2^c)&=& \omega (k_1+ \frac 1 2 (k_1^{'}-k_1),k_2+ \frac 1 2 (k_2^{'}-k_2))\\
&=& \omega(k_1, k_2)+ \frac 1 2 \omega (k_1, (k_2^{'}-k_2))-\frac 1 2 \omega(k_2, (k_1^{'}-k_1))\\
&=& \frac 1 2 \omega (k_1, k_2^{'})-\frac 1 2 \omega (k_2, k_1^{'})\\
&=& 0-0=0. 
\end{eqnarray*}
\qed
}
\remark{In \cite{KS}, the authors produce an example of a symplectic Banach space $V$ in which no subspace is split Lagrangian.  By Theorem \ref{Split}, these subspaces do not admit any closed complements at all.}\\

\definition{A  \emph{ weak canonical relation} $L$ between two weakly symplectic vector spaces $V$ and $W$, denoted by $L\colon V \nrightarrow W $, is a  Lagrangian subspace of $\overline{V}\oplus W$}.

\definition{A \emph{split weak canonical relation} $L$ between two weakly symplectic vector spaces $V$ and $W$, denoted by $L\colon V \nrightarrow W $, is a split Lagrangian subspace of $\overline V\oplus W$}.
\example{Each Lagrangian subspace $L\subseteq V$ may be identified with the canonical relation $\{0\}\times L \subseteq \{0\}\oplus V$} from the zero dimensional vector space $\{0\}$  to $V$.

\example{It is easy to check that, for any weakly symplectic vector space $V$, the diagonal relation 
$$\mbox{graph }( \mbox{Id}_V)=\Delta_V\colon V\nrightarrow V$$
is isotropic and has the anti-diagonal $-\Delta_V$ as an isotropic complement. Thus, $\Delta_V$ is a split canonical relation (as is $-\Delta_V$).
}
\definition{Let $C$ be a coisotropic subspace of $V$ and $\underline C:= C/ C^{\omega}$ its symplectic reduction. The \emph{reduction relation}, denoted by $L_C\colon V\nrightarrow \underline C$, is the graph 
$$L_C:=\{(x, [x])\subset V\oplus \underline C\colon x \in C\}$$ of the of the quotient map from $C$ onto $\underline C$, considered as having source object $V$.}
\proposition{\label{reduction}$L_C$ is a canonical relation.}
\proof{The fact that $L_C$ is isotropic easily follows from
$$\underline {\omega}([x],[y])=\omega (x,y).$$
Furthermore, consider
$$L_C^{\omega}=\{(x,y)\in V\oplus \underline C \colon  \omega(x,x^{'})-\underline{\omega}(y,[x^{'}])=0, \forall (x^{'},[x^{'}])\in L_C\}.$$ Then, when $(x,y) \in L_C^{\omega}$, 
\begin{eqnarray*}
\omega(x,x^{'})-\underline{\omega}(y,[x])=\underline{\omega}([x],[x^{'}])-\underline{\omega}(y,[x^{'}])=\underline{\omega}([x]-y,[x^{'}])=0, \forall [x^{'}] \in \underline C,
\end{eqnarray*}
and since $\underline C$ is symplectic, it implies that $[x]=y$ and hence $(x,y)\in L_C$. \qed}
\remark{As we will se later, $L_C$ is not in general split Lagrangian: a sufficient condition for this is to have suitable complements of $C$ and $C^{\omega}$ in $V$ and $C$ respectively. }

\section{Composition and Reduction} The set-theoretic composition $L_2\circ L_1$   of  canonical relations $L_1 \colon \overline V\nrightarrow W$ and $L_2 \colon \overline W\nrightarrow Z$ is a linear subspace of $\overline V\oplus Z$.  
This composition is strongly transversal (i.e. the projection map from $L_C\times_V L$ onto $V\times Z$ is a closed embedding)  exactly when $L$ is transversal to $C$, where $C=V\oplus \Delta_{W}\oplus Z$.

While, in finite dimensions, this composition is always Lagrangian, and hence a canonical relation $V \nrightarrow Z$, in infinite dimensions it may be   merely isotropic.

In Section \ref{Examples}, we will give two examples where the composition of canonical relations is not Lagrangian.  In the first, the composition is closed; in the second, the weak symplectic structures are symplectic.

\subsection{Equivalence between reduction and composition; split coisotropic subspaces}\label{Composition}
There is a natural correspondence (see e.g. \cite{Weinstein}) between the reduction and composition of canonical relations, that can be spelled out as follows.
For a coisotropic subspace $C$ of a symplectic vector space $V$, and a Lagrangian subspace $L$ of $V$, the reduced Lagrangian subspace $\underline L \subseteq \underline C$, defined as:
$$\underline L:= (L\cap C)/ (L\cap C^{\omega}),$$
can be obtained as the following composition of canonical relations:
$$\underline L= L_C \circ L.$$

On the other hand, the composition of linear canonical relations is itself a particular  case of reduction.  It can be easily checked that the space  
$V\times \Delta_W  \times Z$  is a coisotropic subspace of $\overline V \times  W\times \overline W \times Z$;
and that 
$$(\overline V\times \Delta_W \times Z)^{\omega}= \{0_{\overline V} \} \times \Delta_W \times \{0_Z \},$$
therefore the reduced space
$\underline{\overline V\times \Delta_W\times Z}$ is naturally isomorphic to $V \times \overline Z$. 
Under this natural isomorphism, the composed relation $g\circ f$, where $f: V\nrightarrow W$ and $g: W\nrightarrow Z$ are canonical relations,  is the reduction of the product $f \times g$, and the composition is transversal if and only if the reduction is.

The following theorem gives sufficient conditions for the reduction relation $L_{C}$ to be split Lagrangian.
\theorem{\label{Reduction}Let $C$ be a coisotropic subspace of $V$. If there is a complement $C^c$ to $C$ in $V$ and a complement $C^{'}$  to $C^{\omega}$ in $C$ such that
\begin{enumerate}
\item $C^c$ is isotropic, and
\item $C^{' }$ and $C^c$ are symplectically orthogonal,
\end{enumerate}
then $L_C$ is split Lagrangian.
}
\proof{Given a choice of $C^c$ and $C^{'}$, we have the decomposition
$$V = C^c\oplus C^{\omega} \oplus C^{'}.$$  If we  identify the complement $C^{'} $ in the usual way with the quotient
$ \underline C = C/C^\omega$, we can write the reduction relation as
$$L_C=\{(x_1^c,x_2^{\omega},x_3^{'},[x_4]) \in V\oplus \underline C\colon  x_1^c=0, x_2^{\omega}\in C^{\omega}, x_3^{'}=[x_4] \in C^{'}\}.$$
Consider now the following relation:
$$L_C^{-}:= \{(x_1^c,x_2^{\omega},x_3^{'},[x_4]) \in V\oplus \underline C\colon  x_1^c\in C^{c}, x_2^{\omega}=0, x_3^{'}=-[x_4] \in C^{'}\}.$$
Observe that 
$L_C+L_C^{-}=V\oplus C$ and that $L^C\cap L_C^{-}=\{(0,0,0,0)\}$. By Proposition \ref{reduction}, $L_C$ is isotropic. Now 
$$(L_C^{-})^{\omega}=\{(y_1^{c},y_2^{\omega},y_3^{'},[y_4]) \in V\oplus \underline{\overline C}\colon \omega(y_1^{c}+y_2^{\omega}+y_3^{'}, x_1^{c}+x_3^{'})+\underline \omega[y_4],x_3^{'})=0\colon \forall (x_1^c,0,x_3^{'},-x_3^{'})\in L_C^{-}\}.$$

If $(y_1^c,y_2^{\omega},y_3^{'},[y_4])\in L_C^{-}$, then $y_2^{\omega}=0$ and $y_3^{'}=-[y_4]$, hence
$$\omega(y_1^c+y_2^{\omega}+y_3^{'}, x_1^c+x_3^{'})+\underline \omega([y_4],x_3^{'})=\omega(y_1^c,x_1^c)+\omega(y_1^c,x_3^{'})+\omega(y_3^{'},x_1^{c}).$$
The first term on the right hand side vanishes since $C^c$ is isotropic, by assumption (1). The second and third terms also vanish, by (2). Thus $L_C^{-}$ is isotropic.\\

    
     

 \qed}
\definition\label{SplitCo_Orth}{A coisotropic subspace $C$ is \emph{split coisotropic} if it satisfies the assumptions of Theorem \ref{Reduction} for complements $C^c$ and $C'$.  The triple $(C,C^c,C')$ will be called a \emph{splitting c-triple}.}

We will describe an equivalent formulation of split coisotropic subspaces, by replacing condition (2) with a suitable choice of a split Lagrangian subspace.

\subsection{The finite dimensional case}\label{Finite}
In the case in which $V$ is finite dimensional, the space of split Lagrangians is the same as the full Lagrangian Grassmannian. However, not every Lagrangian subspace in $V \oplus \underline C$ is the reduction relation $L_C$ for some $C$ which is split coisotropic. In this section, we study the finite dimensional case of splitting c-triples and how they are explicitly related with split reduction relations and splitting l-pairs. We also find an equivalent characterization of a split coisotropic subspace $C$ in terms of  split Lagrangian complements of $C^{\omega}$.


Suppose that $V$ is a  $2n$-dimensional vector space and $C$ is a  $n+k$-dimensional coisotropic subspace. This implies that
$\dim(\underline C)=\dim (C/ C^{\omega})= (n+k)-(n-k)=2k.$
 It is easy to observe that $\dim (L_C)=n+k= \dim(L_C^c)$, where $L_C^c$ is a complement to $L_c$ in $V\oplus \overline{\underline C}$.

Now, we observe that the space $\mathcal L$ of Lagrangian complements to $L_C$ is  open in the Lagrangian Grassmannian $\Lambda (n+k)$; therefore it has dimension $d_1= (n+k)(n+k+1)/2$.

We will now determine ``how many" complements in this space are of the form $L_C^{-}$, using the construction from Theorem \ref{Reduction}, i.e. to determine the dimension of the space of such special complements.
 
 To address this question,  we look first at the manifold of splitting c-pairs for $C$.
 Each such pair produces a symplectic decomposition of $V$: 
 $$V= C^{'} \oplus (C^{\omega} \oplus C^c).$$
 
The space $\mathcal C$ of choices of the complement $C^{'}$ is open in the Grassmanian $Gr_{2k}(n+k)$; therefore, it has dimension $d=2k(n-k)$. Now, given a choice of $C^{'}$, choosing a compatible Lagrangian complement $C^{c}$ to $C^{\omega}$ in $(C^{'})^{\omega}$ is equivalent to choosing an element in an open subset of the Lagrangian Grassmanian $\Lambda(n-k)$. Thus, the space of splitting 
c-triples $(C^{'}, C^c)$ (with $C^{'}$ fixed) has dimension $d_2=(n-k)(n-k+1)/2$.
 Now, consider the linear map 
\begin{eqnarray*}
\phi_{C^{'}}: \mathcal C &\to& \mathcal L\\
C^c&\mapsto& C^c\oplus \{0\} \oplus -\Delta_{C^{'}}.
\end{eqnarray*}
It is well defined, due to Theorem \ref{Reduction}, and since$-\Delta_{V}=-\Delta_{W}$ implies that $V=W$, we conclude that $\phi_{C^{'}}$ is 1-to-1.
 
Using the affine structure of both spaces (the space of Lagrangian complements to $L_C$ and the space of splitting c-triples $(C, C^{'}, C^c)$), and the injectivity of $\phi_{C^{'}}$, we conclude that the number of conditions which a Lagrangian complement to $L_C$ must satisfy in order to come from a compatible pair $(C^{'}, C^{c})$ is $d:=d_1-d_2=(2n+1)k$. 
 \remark {We can consider the extreme cases.   When $k=n$, i.e.  $C = V$, we have that $L_C=\Delta_C$, and $d$ in this particular case is $(2n+1) n = 2n(2n+1)/2$ which is the full dimension of 
  $\Lambda(2n)$. 
 This implies that there is a unique compatible pair $(C^{'}, C^{c})$, namely $(V,\{0\})$, which determines uniquely the anti-diagonal $-\Delta_C$ as the relation $L_C^{-}$ in Theorem \ref{Reduction}.
 The other extreme takes place when $k=0$, i.e. $C$ is Lagrangian, in which case $d=0$, i.e. there are no conditions for $L_C^c$. In this particular case, the splitting c-triples $(C, C^{'}, C^c)$ in $V$ are the same as the complements $C^{'}$ of $C$ in $V$, since the reduced space $\underline C$ is $C/C= 0$}.
 

We also observe that since $C$ is coisotropic,  $C^{\omega} \subseteq C$ is an isotropic subspace in $V$, hence also in $(C^{'})^{\omega}$. Therefore, Condition (2) in Definition \ref{SplitCo_Orth} is equivalent to finding an isotropic complement $C^c$ of $C^{\omega}$ in $(C^{'})^{\omega}$, which implies that that $C^{'}$ is a split Lagrangian subspace of $(C^{'})^{\omega}$. 
Notice that $(C^{'})^{\omega}$ is a symplectic subspace of $V$, with dimension $2n-2k$. 
Now, let us consider a Lagrangian complement $C^c$ to $C^{\omega}$ in $(C^{'})^{\omega}$. We prove the following proposition, which implies that $C^c$ is also a complement to $C$ in $V$, symplectically orthogonal to $C^{'}$.

\begin{proposition}
Any lagrangian complement $C^c$ to $C^\omega$ in $(C')^\omega$ is also a complement to $C$ in $V$.
\end{proposition}
\begin{proof}
Let $v \in V$. Then 
\begin{eqnarray*}
v&=& v_{C^{'}}+v_{(C^{'})^{\omega}}\\
&=& v_{C^{'}}+v_{C^{\omega}}+v_{C^c}\\
&=& v_{C}+v_{C^c},
\end{eqnarray*}
where $v_{X}$ denotes the projection of $v$ onto the subspace $X$. Thus $V=C^c+C$.
Now, suppose $v \in C\cap C^c$. We have that 
\[v=v_{C^{'}}\oplus v_{C^{\omega}}\]
since $v\in C$.
Now, for every $w \in C$, $v_{C^{\omega}}$ is symplectically orthogonal to $w$. This implies that $v_{C^{'}}$ is also orthogonal to $w$, thus $v_{C^{'}}\in C^{\omega}$, which implies that $v_{C^{'}}\in C^{'}\cap C^{\omega}= \{0\}$. Therefore $V=C^c\oplus C$, as we wanted.
\end{proof}
This proposition suggests the following equivalent definition of split coisotropic subspace.
\begin{definition}\label{SplitCo_Lag} (Equivalent to Def. \ref{SplitCo_Orth}). A coisotropic subspace $C$ of a weakly symplectic vector space $V$ is split coisotropic if there is a complement $C^{'}$  to $C^{\omega}$ in $C$ and a complement $L$ to $C^{\omega}$ in $(C^{'})^{\omega}$, such that $L$ is isotropic (therefore split Lagrangian).
\end{definition}

 \subsection{Good compositions}


For canonical relations between symplectic manifolds, a clean intersection condition is needed to insure that the composition of two such relations is again a smooth manifold.  In finite dimensions, this condition is automatically satisfied for linear canonical relations.   In infinite dimensions, even in the linear case, we need to impose a condition to insure the good composition of split canonical relations.


\begin{definition}
Let $(L, L^{'})$ be a splitting l-pair of $V$, and let $(C, C^C, C^{'})$ be a splitting c-triple. We say that $(L, L^{'})$ and $C$ {\em intersect neatly} if the following equation holds:
\begin{equation}\label{neat}
C^{'}= L\cap C^{'} + L^{'} \cap C^{'}.
\end{equation}
\begin{remark} From Definition \ref{SplitCo_Lag} it follows that this definition does not depend on the choice of the complement $C^c$.
\end{remark}
\end{definition}


\begin{proposition}
The neatness condition is always satisfied in finite dimensions.  
\end{proposition}
\begin{proof}
If $L$ is Lagrangian, we can consider the intersection $K=L\cap C^{'}$. Let $K^c$ be a complement to $K$ in $C^{'}$. Using the same notation as in Section \ref{Finite}, $\dim K^c \leq 2k$. Thus, there is a choice of a Lagrangian complement $L^{c}$ of $L$ such that $K^c\subseteq L^{'}$. For such choice of complement, the condition in Equation \ref{neat} is satisfied.

\end{proof}

The following proposition gives one sufficient condition for neat intersection in infinite dimensions, namely that the split Lagrangian $L$ is contained in the split coisotropic space $C$.
\begin{proposition} \label{Trans} 
Let $C$ be a split coisotropic subspace of $V$. If $(L, L^C)$ is a splitting l-pair such that $L$ is contained in $C$, then $(L, L^C)$ and $C$ intersect neatly.
\end{proposition}
\begin{proof} If $L^c$ is any isotropic complement to $L$ in $V$,  $L^c\cap C$ is an isotropic complement to $L$ in $C$, due to the modular law for isotropic spaces \cite{Coisotropic}. On the other hand,  $\{0\}$ is a complement to $L\cap C$ in $L$, and these two complements are clearly symplectically orthogonal.
\end{proof}

The following theorem guarantees that the reduction of a split Lagrangian with neat intersection gives rise to a split Lagrangian reduced subspace.

\begin{theorem}\label{NeatRed}
Let $L$ be a split Lagrangian subspace of $V$, intersecting neatly $C$. Then the reduction $\underline L$ is again split Lagrangian.
\end{theorem}

\begin{proof} It is clear that $L_C$ is an isotropic subspace of $\underline C$, since $L\cap C$ is an isotropic subspace of $V$ and thus the induced symplectic structure $\underline \omega$ in $\underline C$ vanishes on $\underline L= L\cap C \Big/ L\cap C^{\omega}$. Now, since $C^{'}\cong \underline C$, the equation \ref{neat} implies that the closed subspace $L^C\cap C^{'}$ is a complement to $\underline L \cong L\cap C^{'}$, and by a similar argument as before this subspace is isotropic.

\end{proof}


Now, we are in a position to impose conditions insuring that the composition of split canonical relations subspaces is again a split canonical relation.\\
\begin{proposition}
\label{Comp}Let $V_1, V_2, V_3$ be three weak symplectic vector spaces. Then the subspace 
$$C:= V_1\oplus \Delta_{V_2} \oplus V_3\subset V_1\oplus \overline {V_2}\oplus V_2\oplus \overline{V_3}$$
is a split coisotropic subspace.
\end{proposition}
\proof{The symplectic orthogonal  
$C^\omega = \{0\}\oplus \Delta_{V_2} \oplus \{0\}$ of $C$ is contained in $C$, which is therefore coisotropic.  To show that it is in fact split coisotropic, we introduce the 
 complement $C^c$ of $C$ in $V_1\oplus \overline {V_2}\oplus V_2\oplus \overline{V_3}$ defined as:
$$C^c:= \{0\}\oplus -\Delta_{V_2} \oplus \{0\}.$$
It is isotropic since $-\Delta_{V_2}$ is isotropic. We also have the  complement $C^{'}$ of $C^{\omega}$ in $C$ defined as:
$$C^{'}=V_1\oplus  \{0\} \oplus V_3.$$
$C^c$ and $C^{'}$ are symplectically orthogonal since 
$$(\omega_1\oplus \overline{\omega_2}\oplus \omega_2 \oplus \overline{\omega_3})((0,v_2,-v_2,0),(v_1,0,0,v_3))=0-0+0-0=0,$$
so all the conditions for a split coisotropic subspace are satisfied.
\qed}

\begin{definition}
Two split canonical relations $L_1: V\nrightarrow W$ and $L_2:W \nrightarrow Z$ are called {\em neatly related} if the split Lagrangian subspace $L_1\oplus L_2$ and the subspace $V\oplus \Lambda_{W}\oplus Z,$ which is split coisotropic by Proposition \ref{Comp}, intersect neatly in $V_1\oplus \overline {V_2}\oplus V_2\oplus \overline{V_3}.$
\end{definition}


\begin{theorem}\label{Compo}
The composition of two neatly related split canonical relations  is a split canonical relation.
\end{theorem}
\proof{ By Theorem \ref{Reduction}, $L_C$ is split Lagrangian, where 
$$C=V_1\oplus \Delta_{V_2} \oplus V_3.$$
Now, the composition $L_2\circ L_1$ is the reduction $\underline{(L_1\times L_2)\cap C}$ of $(L_1\times L_2)\cap C$ in $\underline C= V_1\oplus \overline V_3$.   On the other hand, the product 
if $L_1\colon V_1\nrightarrow V_2$ and $L_2\colon V_2\nrightarrow V_3$ are split, then 
Now, since $(L_1\times L_2)\cap C$ is isotropic, it follows that $\underline{(L_1\times L_2)\cap C}$ is isotropic as well. 
Therefore, invoking Theorem \ref{NeatRed} and the fact that we can phrase composition as a reduction (Section \ref{Composition}), we see that this implies that the space $\underline{(L_1\times L_2)\cap C}$ is split Lagrangian in $V_1\oplus V_3$, as we wanted.
\qed}

\section{Examples}\label{Examples}

\subsection{An example in which the composition is closed but not Lagrangian}
In our first example,\footnote{This example was mentioned briefly in \cite{Con}; here we give more details.} the weak symplectic spaces are Fr\'echet spaces in which we will explicitly construct two canonical relations whose composition is isotropic (and closed) but not Lagrangian. Let
$W := \mathcal C^{\infty}([0,1])$
be the Fr\'echet space of smooth functions on the closed interval $[0,1]$.

Then we have :
\proposition{The form $\omega$ on  $V :=W\oplus W$  defined by 
\begin{equation}
\omega((f\oplus g),(\tilde f\oplus \tilde g))=\int_0^1(f\tilde g-f\tilde g)dt
\end{equation}
is weak symplectic.
}
\proof{The map given by $\omega$ is clearly skew symmetric. To check that it is weakly non degenerate, suppose that there is an element $(f\oplus g)\in V$ such that  $\omega^{\sharp}(f\oplus g)=0$. Set $\tilde f=\tilde g=h$, with $h$ any function in $W$. Then we get that 
$$\int_0^1((f-g)h)dt=0, \forall h \in W,$$
and by the fundamental lemma of variational calculus we conclude that $f=g$. If we set $\tilde f=0$, then 
$$\int_0^1 (f\tilde g)dt=0, \forall \tilde g\in W,$$
and again by the fundamental lemma, we conclude that $f=g=0$, hence $\omega^{\sharp}$ is weak symplectic. \qed
}
\remark{Observe that $\omega$ is not symplectic:  the induced map $\omega^{\sharp}$ is not surjective. The linear dual $V^{*}$ is isomorphic (as a Fr\'echet space) to $\mathcal D([0,1])\oplus \mathcal D ([0,1])$, where $\mathcal D ([0,1])$ denotes the space of distributions on the unit interval.}
The image of $\omega^\sharp$ is a proper dense subspace of this dual.

Consider the following (closed) subspace of $V$ of codimension 1:
$$C:=\{(f\oplus g)\in V\colon g(0)=0\}.$$
\proposition{$C$ is a symplectic and coisotropic subspace of $V$, thus $\underline C= C$.}
\proof{This is equivalent to proving that $C^{\omega}=\{(0,0)\}$. Let $(\tilde f\oplus \tilde g)\in C^{\omega}$. If we set $g\equiv 0$, then we get that 
$$\int_0^1 f\tilde g=0, \forall f \in W,$$
and by the fundamental lemma, we can conclude that $\tilde g\equiv 0$. In a similar way we conclude that 
$$\int_0^1 \tilde fg=0, \forall g\in W \colon g(0)=0,$$
and, again, by the fundamental lemma, $\tilde f\equiv 0$. \qed\\

\begin{remark}
Note that, in this example, the space $C$ is not split coisotropic. The complement to $C^\omega$ in $C$ is all of $C$.  It is symplectically orthogonal only to zero, which is not big enough to contain a complement to $C$ in $V$.   
\end{remark}

Now let us consider the following subspace of $ V$:
\begin{equation}
L=\{ (f\oplus g)\in V\colon  \int_0^1 f dt=0, \frac{dg}{dt}\equiv 0\}.
\end{equation}
}
\proposition{ \label{Lag}L is a Lagrangian subspace of $V$.}
\proof{ $L$ is clearly isotropic.   Now look at
$$L^{\omega}= \{(\tilde f\oplus \tilde g)\in  \colon \int_0^1(f\tilde g-\tilde fg) dt=0~ {\rm whenever}~ \int_0^1 f dt=0 ~{\rm and} ~\frac{dg}{dt}\equiv 0  \}.$$

If we set $f\equiv 0$ and $g\equiv 1$, then we get $\int_0^1 \tilde f\, dt=0$. 

On the other hand, if we set $g \equiv 0$, we get 
$\int_0^1f\tilde g\, dt=0~ {\rm whenever}~ \int_0^1 f dt=0.$
Now let $F(t)$ be any function with $F(0) = 0$ and $F(1) =0,$ and set $f = \frac{dF}{dt},$
so that $\int_0^1 f dt=0 $.
Integrating by parts gives $\int_0^1 F\frac{d\tilde g}{dt} =F\tilde g\Big|_0^1-\int_0^1f\tilde g\, dt=0$
for any $F$ vanishing at $0$ and $1$.  Again by the fundamental lemma, $\frac {d\tilde g}{dt}$ must be identically zero, so we have proved that $\tilde f  \oplus\tilde g$ belongs to $L$ when it is in 
$L^\omega$; i.e. $L$ is Lagrangian. 
\qed
}\\

In fact, we have:

\proposition{$L$ is split Lagrangian}\label{split}.
\proof{Let us consider 
$$L^{'}:=\{(f\oplus g)\colon \frac{df}{dt}\equiv 0, \int_0^1g dt=0 \}.$$ 
Using an argument like that in  Proposition \ref{Lag}, we can check that $L^{'}$ is Lagrangian.  Since 
$$(f\oplus g)=(f-\int_0^1f dt\oplus \int_0^1 g dt)+(\int_0^1 fdt \oplus g-\int_0^1 g dt)\in L + L^{'}$$ and
$$L\cap L^{'}=\{ f\oplus g\colon f\equiv 2\int_0^1f dt,  g\equiv 2\int_0^1g dt\}=\{(0,0)\},$$
we can conclude that $L$ is split Lagrangian.\qed
}\\

Now we consider the composition $L_{C}\circ L$. It is easy to check that
$$L_{C}\circ L=\{f\oplus g \colon \int_0^1 f=0, g\equiv 0\}.$$
\proposition{$L_C\circ L$ is isotropic but not Lagrangian}.
\proof{This composition is clearly isotropic, but any element of $V$ of the form $f\oplus k$, where $k$ is a constant, belongs to $(L_{C}\circ L)^{\omega}.$   As long as $f$ has nonzero integral or $k$ is nozero, this element is not in $ L_{C}\circ L$ itself.    Consequently, the reduced space $(L_{C}\circ L)^{\omega} / L_{C}\circ L$ is 2-dimensional, and the composition is not maximal isotropic. 
 \qed
}

\begin{remark}In this example, the reduced Lagrangian  subspace is finite dimensional. It is an interesting question whether the failure of the composition to be Lagrangian, can be explained in terms of the finite dimensional reduction. \end{remark}

\subsection{An example in which the composition is not closed}
In this example, the composition of two Lagrangian relations between Hilbert spaces fails to be Lagrangian because a projection of a closed subspace of a Hilbert space is not necessarily closed.

Let $\mathcal H = L_2(S^1)$, the Hilbert space of $L_2$-functions of the circle, and let $\mathcal D\colon \mathcal H \to \mathcal H$ be the differentiation operator
\begin{eqnarray*}
\mathcal D\colon \mathcal H& \to& \mathcal H\\
f&\mapsto& \frac{df}{d\theta}.
\end{eqnarray*}

Observe that $\mathcal D$ is a closed operator whose domain is the dense (Sobolev) 
subspace of $\mathcal H$ consisting of functions with distributional derivative in $L_2$.
Therefore, $\mathcal G:= {\rm graph} ~\mathcal D$ is a closed subspace of $\mathcal K:= \mathcal H \oplus \mathcal H$
such that the projection ${\rm pr}_1(\mathcal G)$  to the first component $\mathcal H$ is a dense subspace but not closed. \\
Now, consider $V:= \mathcal K \oplus \mathcal K^*=\mathcal H \oplus \mathcal H \oplus \mathcal H ^* \oplus \mathcal H^*$. This space is symplectic with the canonical form
$$\omega((f_1,g_1, \xi^1\mu^1)(f_2, g_2,\xi^2,\mu^2))=\int_{S^1} \langle f_1, \xi^2 \rangle+\langle g_1, \mu^2 \rangle -\langle f_2,\xi^1\rangle -\langle g_2, \mu^1 \rangle, f_i, g_i\in \mathcal H, \xi^i, \mu^i \in \mathcal H ^*.$$
\begin{proposition} \label{Conormal}
The subspace $\mathcal L:=\mathcal G \oplus \mathcal G^0$ 
is a split Lagrangian subspace of $V$, where $\mathcal G^0$ is the annihilator  of $\mathcal G$ in the dual of $\mathcal K.$
\end{proposition}
\proof{Note that $\mathcal G^0$ is the graph of the adjoint operator $\mathcal D^*$. Since $\mathcal D$ is self adjoint, 
 $\mathcal G = \mathcal G^0$. We also have that $\mathcal G \oplus \mathcal G$ is an isotropic subspace of $V$ (see e.g \cite{Everitt}), and that the orthogonal complement to $\mathcal G$  with respect to the inner product on $\mathcal H\oplus \mathcal H$ is $\mathcal J (\mathcal G)$, where $\mathcal J$ is the symplectic map
\begin{eqnarray*}
\mathcal J\colon \mathcal H \oplus \mathcal H&\to& \mathcal H \oplus \mathcal H\\
(f,g)&\mapsto&(-g,f).
\end{eqnarray*}
It is easy to check that 
this implies that $\mathcal J (\mathcal G)$ is isotropic and hence $\mathcal J (\mathcal G)\oplus \mathcal J (\mathcal G)$ is an isotropic complement to $\mathcal G \oplus \mathcal G ^0$.\qed
}
\begin{remark}
An alternative way to prove Proposition \ref{Conormal} is by considering $\mathcal G \oplus \mathcal G^0$ as the conormal bundle of $\mathcal G$.  We can show that if a subspace $W\subset V$ is split in $V$, then its conormal bundle $N^*W$ is split Lagrangian in $V\oplus V^*$.
\end{remark}

Now consider the space 

$$C= \mathcal H \oplus \mathcal H \oplus \{0\}\oplus \mathcal H^*\subset V.$$
A direct calculation shows that 
$$C^{\omega}=\mathcal H \oplus \{0\}\oplus \{0\}\oplus \{0\} \subset C;$$
hence $C$ is coisotropic and its reduced weak symplectic space is $\underline C=\mathcal H \oplus  \mathcal H^*$.\\

\begin{remark}Note that, in this example, the space $C$ is split coisotropic. A complement to $C$ in $V$ is given by the isotropic subspace
$$\{0\}\oplus \{0\} \oplus \mathcal H ^* \oplus \{0\}.$$  A complement to $C^\omega$ in $C$ is given by the symplectic subspace
$$\{0\} \oplus \mathcal H \oplus \{0\}\oplus \mathcal H^*.$$   These are indeed symplectically orthogonal.
\end{remark}
In order to describe the reduction $L_C(L) = \underline L$ of $L$ in $\underline C$, we compute the intersection
\begin{eqnarray*}
\mathcal L \cap C&=&\mathcal G \oplus \mathcal G^*\cap \mathcal H \oplus \mathcal H \oplus \{0\}\oplus \mathcal H^*\\
&=&\mathcal G\oplus {\{0\}\oplus \{0\}}. 
\end{eqnarray*}
It follows that the projection of $\mathcal L \cap C$ onto $\underline C$ is $(pr_1(\mathcal G),0)=(\mathfrak {Dom}(\mathcal D),0)$. Since the domain of $\mathcal D$ is a not closed subspace of $\mathcal H$, it follows that the reduction of $\mathcal L$ is isotropic but not Lagrangian.\\






\subsection{An example of a maximally isotropic subspace that is not split Lagrangian}
In \cite{KS}, a strong symplectic reflexive Banach space for which there are no split Lagrangian subspaces is constructed. The example in consideration, denoted by $Z_2$,  can be seen as a \emph{twisted sum} of two Hilbert spaces isomoprhic to $l_2$. We will show that $Z_2$ has a closed maximally isotropic subspace $M$ which is maximally isotropic and then by \cite{KS} Theorem 10 we can conclude that such space is not split Lagrangian.\\

The construction of $Z_2$ goes as follows. Given any real sequence $x=\{x_n \}_{n=1}^{\infty}$, its even subsequence $y=Ex$ is defined as $y_n=x_{2n}$, for all $n$. Now, if $x$ is a sequence in $l_2$, we define the sequence $\phi(x)$ as
\begin{eqnarray*}
\phi_{2n}(x)&=&x_n\\
\phi_{2n-1}(x)&=&x_n \log \frac{\vert\vert x \vert \vert_2}{\vert x_n \vert}, n\geq 1.
\end{eqnarray*}
$Z_2$ is then defined as the space of sequences $x$ such that
\begin{equation*}
\vert\vert  x \vert \vert:= \vert \vert Ex \vert \vert_2 + \vert \vert x- \phi(Ex)\vert \vert_2 < \infty.
\end{equation*}
In \cite{KS} is proven that $Z_2$ is indeed a reflexive Banach space, equipped with the symplectic form 
\[\Omega(x,y)= \sum_{n=1}^{\infty} (x_{2n}y_{2n-1}-x_{2n-1}y_{2n}).\]
Furthermore, it is proven that the standard unit vectors $\{e_n\}n+1^{\infty}$ in $\mathbb R^{\infty}$ form a basis of $Z_2$.
Now, the subspace $M$ is defined as the space generated by the odd unit vectors:
\[M =\langle  e_1, e_3, \cdots e_{2k+1}, \cdots  \rangle\]
and it is in fact a closed subspace of $Z_2$ isomorphic to $l_2$.
\begin{proposition}
$M$ is maximally isotropic.
\end{proposition}
\begin{proof}
We compute directly the symplectic orthogonal complement of $M$:
\[M^{\omega}=\{\{z\}_{n=1}^{\infty}: x_{2n}z_{2n-1}-x_{2n-1}z_{2n}=0, \forall x \in M \}.\]
Since $x_{2n}=0$ for sequences in $M$, then it follows that $x_{2n-1}z_{2n}=0$, for all $n\geq 1$. This implies that $z_{2n}=0$ and therefore
\[M^{\omega}= \langle e_1, e_3, \cdots, e_{2n-1}, \cdots \rangle= M\]
and thus $M$ is maximally isotropic.
\end{proof}
\begin{remark}
In the introduction of \cite{KS} it is stated without proof that examples of maximally isotropic non split Lagrangians arise from the conormal bundle of uncomplemented closed subspaces of a reflexive Banach space $X$. It is still unclear to us why this is the case.
\end{remark}
The following section provides an example of a topological field theory, the Poisson sigma model, in which split Lagrangian and split coisotropic spaces arise naturally, describing the evolution relations and Cauchy data respectively. For such spaces, the conditions of Theorems   \ref{Reduction} and \ref{Compo} are satisfied; therefore, they obey the rules of the theory of reduction and composition of split canonical relations. As explained in the Introduction (for more details see e.g \cite{Lagrangian}), the evolution relations obtained from the symplectic formulation of such theories are always isotropic, and given \emph{good} boundary conditions, they come equipped with closed isotropic complements. 
In Section \ref{PSM} we will explicitly construct such isotropic complements, and we will briefly discuss an application to the connection between symplectic groupoids and Poisson structures.

\section{Split Lagrangian subspaces in Lagrangian field theories} \label{PSM}
\subsection{Symplectic formulation}
In Lagrangian field theories, from the perspective of $BV-BFV$ theories with boundary \cite{Lagrangian}, the space of boundary fields $\mathcal F_{\partial M}$ is an infinite dimensional weak symplectic manifold. The symplectic form $\omega_{\partial M}$ is usually an exact 2-form whose primitive 1-form $\alpha$ can be constructed from the boundary contribution of the variational problem for the action $\mathcal S$ of the theory, that is a function on the space of bulk fields $\mathcal F_{M}$. The theory comes equipped with a surjective submersion $\pi\colon \mathcal F_M\to \mathcal F_{\partial M}$.

\subsubsection{The evolution relation}
Given $EL(M)$, the space of solutions of the Euler-Lagrange equations on $M$, the projection $L_M:=\pi(EL(M))$ of such solutions represents the analogue of the Hamiltonian flow for finite dimensional classical systems, although it does not correspond in general to a graph of a symplectomorphism on $\mathcal F_{\partial M}$. Such projections are called \emph{evolution relations}. In \cite{Lagrangian},  it is proven that 
\begin{theorem} $L_{M}$ is an isotropic subspace of $\mathcal F_{\partial M}$.
\end{theorem}
For reasonably \emph{good} Lagrangian field theories (Scalar field theory, abelian $BF$-theory and Yang-Mills theory for instance ), the evolution relations are Lagrangian. There are examples, such as the case of the wave equation ( see e.g. \cite{Wave}), in which the evolution relation is not always Lagrangian. In \cite{Con} it was proven that the evolution relations associated to  Poisson sigma models where the world-sheet has zero genus are Lagrangian. We will prove that in fact, such evolution relations are split Lagrangians. We will first prove the condition in the cases where the Poisson structure is trivial or constant, and use a change of coordinates argument, to show that the other cases can be reduced to the former.
\subsection{Linear Poisson sigma models}
A detailed description of the symplectic formulation of the Poisson sigma model as a Lagrangian field theory with boundary can be found in \cite{Relational, Symplectic, Con}. 
\subsection{Zero Poisson structure} We consider by simplicity the one dimensional linear case, i.e. $M=\mathbb R$, equipped with the zero Poisson structure \footnote{ The arguments in the sequel can be directly generalized to the case where $M=\mathbb R^n$.} . We will restrict to the case where the worldsheet is a disk, denoted by $D_n$,  with the boundary $S^1$ split into $2n$ closed intervals $I$ intersecting at the end points. On alternating intervals we impose boundary conditions in such way that the remaining $n$ intervals are free, on for each one of them we associate the weak symplectic space $V$ of boundary fields:
$$V=T^{*}(\mathcal C^{\infty}([0,1], \mathbb R)),$$
naturally identified with the cotangent bundle of the path space of $M$, understood as follows:
$V$ has two canonical coordinates $(X,\eta)$ (position and momentum), given by
\begin{eqnarray*}
X&\in& \mathcal C^{\infty}([0,1], \mathbb R),\\
\eta &\in& \Gamma^{\infty}(T^{*}[0,1]\oplus X^{*}(T^{*}\mathbb R))\cong \mathcal C^{\infty}([0,1], \mathbb R^*\cong \mathbb R).
\end{eqnarray*}
The boundary condition for the alternating non free intervals is given by $\eta=0$.
\begin{remark}
In this version of the PSM, we consider the space of boundary fields to be paths without constraints on the initial and final points. The case of fields with initial and final constraints, i.e. the paths start and end in coisotropic submanifolds of $M$ (see e.g. \cite{Dual}), and periodic fields (loops instead of paths), will be consider in a subsequent work.
\end{remark}
\begin{figure}[ht]
    \centering
    \def\svgwidth{200 pt}
    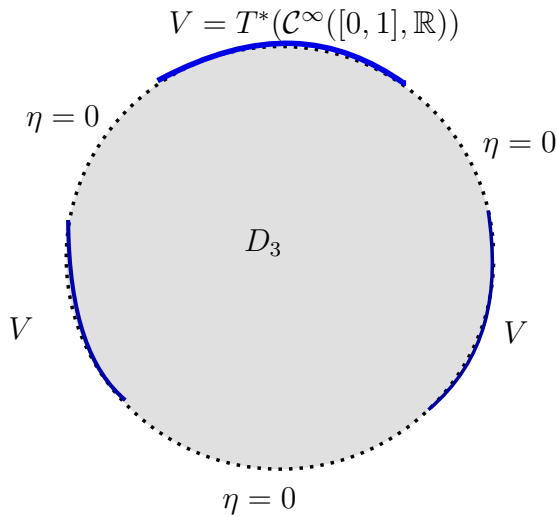
    \caption{The boundary fields for PSM with worldsheet $D_3$. The blue segments represent the free intervals for which we associate the weak symplectic space $V$. For the dashed segments we associate fields with boundary condition $\eta=0$. }
\end{figure}

We first observe that the subspace of solutions of Euler-Lagrange equations is described by
\begin{equation}\label{EL}C_M=\{(X,\eta)\colon \frac{dX}{dt}\equiv0\}\cong \mathbb R \oplus \mathcal C^{\infty}([0,1], (\mathbb R)^{*}).\end{equation}
Its symplectic complement is given by
$$C_M^{\omega}=\{(X,\eta)\colon X=0, \int_0^1\eta=0\}.$$
The symplectic reduction is 
$$\underline {C_M}= C_M/ C_M^{\omega}= \mathbb R\oplus (\mathbb R)^*\cong T^{*}(\mathbb R).$$
\begin{proposition} The coisotropic subspace $C_M$ is split.
\end{proposition}
\proof{Consider the following complement to $C_M$ in $V$:
$$C_M^C:= \mathcal C_0^{\infty}([0,1], \mathbb R)\oplus \{0\},$$
where 
$$\mathcal C_0^{\infty}([0,1], \mathbb R):=\{ X \in \mathcal C^{\infty}([0,1], \mathbb R)\colon \int_0^1 X(t)dt=0\}.$$
It is a complement since the two subspaces are disjoint and any pair $(X,\eta)$ in  $V$ can be written as
$$(X,\eta)=(\int_0^1 X(t) dt,\eta)+(X-\int_0^1 X(t) dt, 0)$$
with the first component in $C_M$, the second component in $C_M^C.$
For two elements $(X,0)$ and $(X^{'},0)$ in $C_M^{C}$, we get that 
$$\omega (((X,0),(X^{'},0))=\int_0^1 (0\cdot X(t) -0\cdot X^{'}(t))\, dt=0,$$
Therefore $C^C_M$ is an isotropic complement to $C_M$.\\
Now consider the following complement to $C_M^{\omega}$ in $C_M$:
$$C_M^{'}:=\{(X,\eta)\colon \frac {dX}{dt}\equiv0,\frac{d\eta}{dt}\equiv0\}\cong \mathbb R \oplus \mathbb R^{*}.$$ 
It is a complement since the two subspaces are disjoint any pair $(X,\eta)$ in  $C_M$ can be written as
$$(X,\eta)=(0, \eta-\int_1^0 \eta \, dt)+(X, \int_0^1 \eta\, dt)$$
with
$$(0, \eta-\int_1^0 \eta \, dt)\in C_M^{\omega}, (X, \int_0^1 \eta\, dt)\in C_M^{'}.$$
Now, if  $(X,\eta) \in C_M^C$ and $(X^{'},\eta^{'})$ in $C_M^{'}$, we get that 
$$\omega (((X,\eta),(X^{'},\eta^{'}))=\int_0^1 -(0\cdot X^{'}(t) +\eta^{'} \cdot X(t))\, dt=\eta^{'} \int_0^1 X(t)\, dt=0,$$
therefore $C^C_M$ and $C_M^{'}$ are symplectically orthogonal, hence $C_M$ is a split coisotropic subspace, as we wanted.
\qed}
 \subsection{The evolution relations} 
 \subsubsection{The relation $L_M^1$}
Now, consider the disk $D_1$ and the following subspace $L^1_M$ of $C_M$:
$$L^1_M:= \{(X,\eta)\colon \frac{d X}{dt}\equiv 0, \int_0^1 \eta(t) \, dt=0\}.$$
\begin{proposition}
$L^1_M$ is a split Lagrangian subspace of $V$ and $L^1_M$ and $C_M$ intersect neatly
\end{proposition}
\proof{ To check that $L_M^1$ is isotropic, if $(X,\eta)$ and $(X^{'}, \eta^{'})$ are in $L_M^{1}$, then
$$\omega((X,\eta),(X^{'},\eta^{'}))=X\int_0^1\eta^{'}(t)\, dt -X^{'}\int_0^1\eta(t)\, dt=0-0=0.$$

Consider the following complement to $L^1_M$ in $V$:
$$(L^1_M)^C:=\{(X,\eta)\colon  \frac{d \eta}{dt}\equiv 0, \int_0^1 X(t)\, dt=0\}.$$
It is isotropic since, if $(X,\eta)$ and $(X^{'}, \eta^{'})$ are in $(L_M^{1})^C$, then
$$\omega((X,\eta),(X^{'},\eta^{'}))=\eta\int_0^1X^{'}(t)\, dt -\eta^{'}\int_0^1 X(t)\, dt=0-0=0.$$
The neat intersection condition follows from Proposition \ref{Trans}.
\qed}

In this case we obtain that the reduced space $\underline{L^1_M}$ is the zero section $\mathbb R$ of $\mathbb R^{*}\oplus \mathbb R$, that is clearly split Lagrangian (it is a Lagrangian subspace of a finite dimensional symplectic vector space).
 \subsubsection{The relation $L_M^2$}

In a similar way as before, we consider the disk $D_2$ for which we can define the following relation $L^2_M\colon V \nrightarrow V$:
$$L^2_M:=\{(X,\eta, X^{'}, \eta^{'})\subset V\oplus V\colon X=X^{'}=\mbox{ct}, \int_0^1 (\eta-\eta^{'})(t)\, dt =0\},$$
where ct is a constant.
\begin{proposition}
$L_M^2$ is a split Lagrangian subspace of $V\oplus \overline V$, where $L_M^2$ and $C_M\oplus C_M$ intersect neatly.
\end{proposition}
\proof{ 
To check that $L_M^2$ is isotropic, if $(X_1,\eta_1,X_1,\eta_1^{'})$ and $(X_2,\eta_2,X_2,\eta_2^{'})$ are in $L_M^{2}$, then
\begin{eqnarray*}
&&\omega\oplus \overline{\omega}\,((X_1,\eta_1, X_1, \eta_1^{'}),(X_2,\eta_2, X_2, \eta_2^{'}))\\
&=&X_1\int_0^1 \eta_2(t)\, dt-X_2\int_0^1\eta_1(t)\, dt -X_1\int_0^1 \eta_2^{'}(t)\, dt+X_2\int_0^1\eta_1^{'}(t)\, dt\\
&=&X_1\int_0^1(\eta_2-\eta_2^{'})(t)-X_2\int_0^1(\eta_1-\eta_1^{'})(t)\, dt\\
&=&0-0=0.
\end{eqnarray*}

Now we consider the following subspace of $V \oplus \overline V$:

$$(L^2_M)^C:=\{(X,\eta, X^{'}, \eta^{'})\subset V\oplus V\colon \int_0^1(X+X^{'})(t)\, dt= 0, \eta=-\eta^{'}, \eta=\mbox{ ct}\}.$$
To check that  $(L^2_M)^C$ is a complement to $L^2_M$, we realize that any element of the space $V \oplus \overline V$ can be decomposed as
$$(X,\eta,X^{'}, \eta^{'})= $$
$$(\frac{\int_0^1 (X+X^{'})(t)\, dt}{2}, \eta+ \frac{\int_0^1 (\eta-\eta^{'})(t)\, dt}{2}, \frac{\int_0^1 (X+X^{'})(t)\, dt}{2},  \eta^{'}+ \frac{\int_0^1 (\eta^{'}-\eta)(t)\, dt}{2}  ) +$$
$$(X-\frac{\int_0^1 (X+X^{'})(t)\, dt}{2}, \frac{\int_0^1 (\eta^{'}-\eta)(t)\, dt}{2}, X^{'}-\frac{\int_0^1 (X+X^{'})(t)\, dt}{2}, \frac{\int_0^1 (\eta-\eta^{'})(t)\, dt}{2}).$$
The first summand belongs to $L^2_M$ and the second one belongs to $(L^2_M)^C$ and it is a direct check that the summands are disjoint.
To check that $(L_M^2)^C$ is isotropic, we check that, for two elements $(X_1,\eta_1, X^{'}_1, -\eta_1^{'})$ and $(X_2,\eta_2, X^{'}_2, -\eta_2^{'})$ 
\begin{eqnarray*}
&&\omega\oplus \overline{\omega}\,((X_1,\eta_1, X^{'}_1, -\eta_1),(X_2,\eta_2, X^{'}_2, -\eta_2))\\
&=&\eta_2\int_0^1 X_1(t)\, dt-\eta_1\int_0^1X_2(t)\, dt+\eta_2\int_0^1 X_1^{'}(t)\, dt-\eta_1\int_0^2X_2^{'}(t)\, dt\\
&=&\eta_2\int_0^1(X_1+X^{'}_1)(t)-\eta_1\int_0^1(X_2+X^{'}_2)(t)\, dt\\
&=&0-0=0.
\end{eqnarray*}
\qed}
\newline
We can verify that the reduced relation $\underline{L_2}$ is the diagonal relation $\Delta_{\mathbb R^{*} \oplus \mathbb R}\colon \mathbb R^{*} \oplus \mathbb R\nrightarrow \mathbb R^{*} \oplus \mathbb R$ that is split Lagrangian.

 \subsubsection{The relation $L_M^3$}
 For $L_M^3$ we consider the concatenation (composition) of paths the $X$- coordinate, namely, if $ X(1)=X^{'}(0)$
 
 $$
X*X^{'}(t)= \left\{ \begin{array}{rl}
 X(2t) &\mbox{ if $0\leq t \leq \frac 1 2 $} \\
  X^{'}(2t-1) &\mbox{ if $\frac 1 2 \leq t \leq 1 $}
       \end{array} \right.
       $$
Note that in general, $*$ is not a well defined operation on $V$, since the concatenation of two smooth paths is not in general smooth and since we have not defined the operation for the $\eta$-coordinate.
However, as proven in \cite{Symplectic}, given two composable paths $X$ and $X^{'}$ such that $(X,\eta)$ and $(X^{'}, \eta^{'})$ in $C_M$, there are pairs $(\tilde X, \tilde{\eta})$ and $(\tilde X ^{'}, \tilde{\eta}^{'})$ in $C_M$ such that

\begin{enumerate}
\item $((X,\eta),(\tilde X, \tilde{\eta}))$ and $(X^{'}, \eta^{'}),(\tilde X ^{'}, \tilde{\eta}^{'})$ are elements of $L_M^2$.
\item The paths $\tilde{\eta}$ and $\tilde{\eta^{'}}$ are composable, furthermore, $\tilde{\eta(0)}=\tilde{\eta}(1)=\tilde{\eta}^{'}(0)=\tilde{\eta}^{'}(1)=0$.
\item $(\tilde X, \tilde{\eta})*(\tilde X ^{'}, \tilde{\eta}^{'}):=(\tilde X*\tilde X^{'}, \tilde{\eta}*\tilde{\eta}^{'}) \in C_M$.
\end{enumerate}

Therefore, for the next evolution relation, we can assume without loss of generality that the path concatenation is smooth. We consider the following relation
$$L_M^3:=\{(X,\eta,X^{'}, \eta^{'}, X^{''}, \eta^{''})\in C_M\oplus C_M \oplus C_M\colon ((\tilde X*\tilde X^{'}, \tilde{\eta}*\tilde{\eta}^{'}), (X^{''}, \eta^{''}))\in L^2_M\},$$
that can be described by the following constraints:
\begin{eqnarray*}
&&X=X^{'}=X^{''}=\mbox{ ct },\\
&&\int_0^1\eta^{''}(t)\, dt= \int_0^1(\eta +\eta^{'})(t)\,dt.
\end{eqnarray*}
\begin{proposition}
$L^3_M$ is a split Langrangian subspace of $V\oplus V \oplus \overline V$ intersecting neatly$C_M\oplus C_M \oplus C_M$.
\end{proposition}

\proof{ To check that $L^3_M$ is isotropic, if we consider to elements $(X_1,\eta_1,X_1, \eta_1^{'}, X_1, \eta_1^{''})$ and $(X_2,\eta_2,X_2, \eta_2^{'}, X_2, \eta_2^{''})$ in $L^3_M$, then 
\begin{eqnarray*}
&&\omega\oplus \omega\oplus \overline{\omega}\,((X_1,\eta_1,X_1, \eta_1^{'}, X_1, \eta_1^{''}),(X_2,\eta_2,X_2, \eta_2^{'}, X_2, \eta_2^{''}))\\
&=&X_1\int_0^1 \eta_2(t)\, dt-X_2\int_0^1\eta_1(t)\, dt +X_1\int_0^1 \eta_2^{'}(t)\, dt-X_2\int_0^1\eta_1^{'}(t)\, dt\\
&&- X_1\int_0^1\eta_2^{''}(t)\, dt+X_2\int_0^1 \eta_1^{''}(t)\, dt\\
&=&X_1\int_0^1(\eta_2+\eta^{'}_2-\eta_2^{''})(t)-X_2\int_0^1(\eta_1+\eta_1^{'}-\eta_1^{''})(t)\, dt\\
&=&0-0=0.
\end{eqnarray*}

Now, let us consider the following complement $(L^3_M)^C$ of $L^3_M$ in $V\oplus V \oplus \overline V$:

 $$
(L^3_M)^C:=\left\{ \begin{array}{lr}
(X,\eta,X^{'}, \eta^{'}, X^{''}, \eta^{''})\in V\oplus V \oplus \overline V\colon &\\
  \eta=\eta^{'}=-\eta^{''}=\mbox{ ct },&\\
  \int_0^1(X +X^{'}+X^{''})(t)\,dt=0.
       \end{array} \right\}.
$$
To check that it is a complement, we consider the following decomposition of any element $(X,\eta,X^{'}, \eta^{'}, X^{''}, \eta^{''})$ in  $V\oplus V \oplus V$, where $\sigma_3(X):= \frac 1 3 \int_0^1 (X+X^{'}+X^{"})(t)\, dt$ and $\sigma_3(\eta):=\frac 1 3 \int_0^1 (\eta+\eta^{'}+\eta^{"})(t)\, dt$:
\begin{eqnarray*}
&&(X,\eta,X^{'}, \eta^{'}, X^{''}, \eta^{''})=\\
&&(\sigma_3(X),\eta-\sigma_3(\eta),\sigma_3(X), \eta^{'}-\sigma_3(\eta), \sigma_3(X), \eta^{''}+\sigma_3(\eta))\\
&+& (X-\sigma_3(X),\sigma_3(\eta),X^{'}-\sigma_3(X),\sigma_3(\eta), X^{''}-\sigma_3(X), -\sigma_3(\eta)).
\end{eqnarray*}
It can be checked that the summands are disjoint.

To check that $(L^3_M)^C$ is isotropic, if we consider to elements $(X_1,\eta_1,X_1^{'}, \eta_1, X_1^{''}, -\eta_1)$ and $(X_2,\eta_2,X_2^{'}, \eta_2, X_2^{''},- \eta_2)$ in $L^3_M$, then 
\begin{eqnarray*}
&&\omega\oplus \omega\oplus \overline{\omega}\,((X_1,\eta_1,X_1^{'}, \eta_1, X_1^{''}, -\eta_1),(X_2,\eta_2,X_2^{'}, \eta_2, X_2^{''}, -\eta_2))\\
&=&\eta_2\int_0^1 X_1(t)\, dt-\eta_1\int_0^1X_2(t)\, dt +\eta_2\int_0^1 X^{'}_1(t)\, dt-\eta_1\int_0^1X^{'}_2(t)\, dt\\
&&- \eta_1\int_0^1X_2^{''}(t)\, dt+\eta_2\int_0^1 X_1^{''}(t)\, dt\\
&=&\eta_2\int_0^1(X_1+X^{'}_1+X_1^{''})(t)-\eta_1\int_0^1(X_2+X_2^{'}+X_2^{''})(t)\, dt\\
&=&0-0=0.
\end{eqnarray*}
\qed}
\newline
The reduced relation $\underline {L_3}$ can be described as
$$\underline{L_3}=\{((\mu_1,x),(\mu_2,x),(\mu_1+\mu_2,x)), x\in \mathbb R, \mu_i \in \mathbb R ^{*}\}= \mbox( graph )(+_x)\colon T^{*}\mathbb R \oplus T^{*}\mathbb R\nrightarrow T^{*}\mathbb R,$$
where $+_x$ denotes the fiber addition in the cotangent bundle of $\mathbb R$. It can be checked directly that this is a split canonical relation.

\subsection{The symplectic case} We will consider now the case where $M=\mathbb R^2$ and $\Pi=\Omega^{-1}=
  \begin{bmatrix}
    0 & 1 \\
    -1& 0
  \end{bmatrix}
$.
In this case, the space of solutions of Euler-Lagrange equations is given by
\begin{eqnarray*}
C_M&=& \{(X,\eta)\colon \frac{dX}{dt}=\Omega^{-1}\eta\}\\
&=&\{(X, \Omega \frac{d X}{dt}), X \in \mathcal C^{\infty}([0,1], \mathbb R^2)\}\\
&\cong& C^{\infty}([0,1], \mathbb R^2).
\end{eqnarray*}
\begin{proposition}
$C_M$ is split coisotropic.
\end{proposition}
\proof{ We compute the symplectic orthogonal subspace $C_M^{\omega}$:
$$C_M^{\omega}=\{((X,\eta) \in V\colon X \in \mathcal C_0^{\infty}([0,1], \mathbb R^2), \eta=\Omega \frac{dX}{dt}) \} \subset C_M,$$
where 
$$C_0^{\infty}([0,1], \mathbb R^2):=\{X \in C^{\infty}([0,1], \mathbb R^2)\colon X(0)=X(1)=0\}.$$
The symplectic reduction is 
$$\underline{C_M}=C_M/ (C_M)^{\perp}= \mathbb R^2\oplus \mathbb R^2.$$
Now consider the following complement $(C_M)^C$ of $C_M$ in V:
$$(C_M)^C:=\{(X,0)\colon \int_0^1 X(t)\, dt=0\}.$$
It is isotropic since the $\eta$-component of $(C_M)^C$ is always zero. To check that it is a complement, we consider the following disjoint decomposition of $(X,\eta)$ in $V$: If $\tilde X:= X(t)-\omega^{-1} \, \int_0^t \eta(t) \,dt$,  then
$$(X,\eta)= (\tilde X - \int_0^1 \tilde X (t)\, dt, 0)+ (\Omega^{-1}(\int_0^t \eta(t)\, dt+ \int_0^1\eta(t)\, dt), \eta)$$
Now consider the following complement $(C_M)^{'}$ of $(C_M)^{\omega}$ in $C_M$:
$$(C_M)^{'}:=\{(X_0+ X_1t, \Omega X_1)\colon X_0, X_1 \in \mathbb R^{2}\}.$$
It is a complement since the following is a disjoint decomposition of elements $(X, \Omega \frac{d X}{dt})$ in $C_M$:
\begin{eqnarray*}
(X, \Omega \frac{d X}{dt})&=&(X-(X(0)-(X(1)-X(0))t), \Omega (\frac{dX}{dt}-(X(1)-X(0)))\\
&+& (X(0)-(X(1)-X(0))t, \Omega (X(1)-X(0))),
\end{eqnarray*}
where the first summand is in $(C_M)^{\omega}$ and the second summand is in $(C_M)^{'}$.\\
Now, we check that $(C_M)^{C}$ and $(C_M)^{'}$ are symplectically orthogonal:
$$ \omega ((X, 0), (X^{'}_0+ X^{'}_1t, \Omega X^{'}_1))=\Omega X^{'}_1\int_0^1 X(t)\, dt=0.$$

\qed
}

The following is the relation $L^1_M$ for this case:
$$L_M^1=\{(X, \Omega \frac{d X}{dt})\colon X(0)=X(1)\}.$$ 
\begin{proposition} $L^1_M$ is split Lagrangian intersecting neatly $C_M$.
\end{proposition}
\proof{ $L^1_M$ is isotropic since 
$$\omega((X, \Omega \frac{d X}{dt}), (X^{'}, \Omega \frac{d X^{'}}{dt}))=X\Omega X^{'}\vert_0^1=0.$$
Its isotropic complement $(L_M^1)^C$ is given by
$$(L_M^1)^C:=\{(X,\mbox{ ct})\colon\int_0^1 X(t)\, dt=0\}.$$
It is clearly isotropic. We verify that it is in fact a complement. If we denote 
$$\tilde{\eta}(t):= \int_0^t (\eta(t)\, dt - \int_0^1 \eta (t))\, dt$$ and
$$\tilde X= \int_0^1 X(t)- \omega^{-1} \tilde{\eta}(t)\, dt$$
 $(X,\eta)$ can be decomposed as
$$(X,\eta)=(\omega^{-1} \, \tilde {\eta}+\tilde X, \eta- \int_0^1 \eta (t)\, dt)+(X- \omega^{-1} \tilde{\eta}-\tilde X , \int_0^1 \eta (t))\, dt),$$
in which the first component belongs to $L^1_M$ and the second one to $(L_M^1)^C$. It is disjoint since
\begin{eqnarray*}
&&\eta- \int_0^1 \eta (t)=\int_0^1 \eta (t)\iff  \eta \equiv 0; \\
&&\omega^{-1} \, \tilde {\eta}+\tilde X=X- \omega^{-1} \tilde{\eta}-\tilde X \iff X=2\int_0^1 X(t)\, dt\iff X\equiv 0.
\end{eqnarray*}

\qed}

The following is the description of $L^2_M$ in the symplectic case:
$$L_2=\{(X, \Omega \frac{dX}{dt}),(X^{'}, \Omega \frac{dX^{'}}{dt})\colon X(0)=X^{'}(0), X(1)=X^{'}(1)\}.$$
\begin{proposition} $L^2_M$ is a split Lagrangian subspace intersecting neatly  $C_M\oplus C_M$
\end{proposition}
\proof{ 
$L^2_M$ is isotropic since 
\begin{eqnarray*}
&&\omega \oplus \overline{\omega}(X_1, \Omega \frac{dX_1}{dt},X^{'}_1, \Omega \frac{dX^{'}_1}{dt},X_2, \Omega \frac{dX_2}{dt},X_2^{'}, \Omega \frac{dX_2^{'}}{dt})\\
&&=\int_0^1 X_1(t)\Omega  \frac{dX_2}{dt} -\int_0^1 X_2(t)\Omega  \frac{dX_1}{dt}-\int_0^1 X_1^{'}(t)\Omega  \frac{dX_2^{'}}{dt}+\int_0^1 X_1^{'}(t)\Omega  \frac{dX_2^{'}}{dt}\\
&&= X_1(1)\Omega X_2(1)-X_1(0)\Omega X_2(0)-X^{'}_1(1)\Omega X_2^{'}(1)+X_1^{'}(0)\Omega X_2^{'}(0)=0.
\end{eqnarray*}
Now, we consider the following isotropic complement $(L^2_M)^C$ of $L^2_M$ in $V \oplus \overline V$:
$$(L^2_M)^C:=\{(X,\mbox{ ct }, X^{'}, \mbox{ ct} ^{'})\colon \int_0^1 X(t)\, dt=\int_0^1 X^{'}(t)\,dt=0\}. $$ 
It is clearly isotropic; in order to check that it is a complement, we consider the following decomposition of $V\oplus \underline V$: If $(X, \eta, X^{'}, \eta^{'})$ is an element of $V \oplus \overline {V}$ and if we set
\begin{eqnarray*}
\tilde {\eta}&:=& \eta- \int_0^1 \eta(t)\, dt,\\
\tilde {\eta^{'}}&:=& \eta^{'}- \int_0^1 \eta^{'}(t)\, dt,\\
\tilde{X} &:=& \int_0^1 (X-\Omega^{-1}\int_0^t \tilde{\eta})(t)\, dt,\\
\tilde{X^{'}} &:=& \int_0^1 (X^{'}-\Omega^{-1}\int_0^t \tilde{\eta^{'}})(t)\, dt,
\end{eqnarray*}

then 
\begin{eqnarray*}
(X,\eta, X^{'}, \eta^{'})= (\Omega^{-1}\int_0^t \tilde {\eta}(t)\, dt + \tilde X, \eta-\int_0^1 \eta(t)\, dt, \Omega^{-1}\int_0^t \tilde {\eta^{'}}(t)\, dt + \tilde {X^{'}}, \eta^{'}-\int_0^1 \eta^{'}(t)\, dt)\\
+(X-\Omega^{-1}\int_0^t \tilde {\eta}(t)\, dt - \tilde X,\int_0^1 \eta(t)\, dt,  X^{'}-\Omega^{-1}\int_0^t \tilde {\eta^{'}}(t)\, dt - \tilde {X^{'}},\int_0^1 \eta^{'}(t)\, dt ).
\end{eqnarray*}
It is a disjoint decomposition since
\begin{eqnarray*}
&&\eta-\int_0^1\eta(t)\,dt=\int_0^1\eta(t)\,dt\iff \eta \equiv 0\iff \tilde{\eta}\equiv 0\iff \tilde X=\int_0^1 X(dt),\\
&& X-\Omega^{-1}\int_0^t \tilde {\eta}(t)-\tilde X=\Omega^{-1}\int_0^t \tilde {\eta}(t)\, dt+ \tilde X \iff X=2\int_0^1 X(t)\,dt \iff X\equiv 0.
\end{eqnarray*}

\qed}

The following is a description of $L^3_M$ in the symplectic case:
\begin{eqnarray*}
L^3_M&=&\{X,\Omega(\frac{dX}{dt}),X^{'}, \Omega(\frac{dX^{'}}{dt}),X^{''}, \Omega(\frac{dX^{''}}{dt})\in C_M\oplus C_M \oplus C_M\colon\\
&&X(1)=X^{'}(0), X^{''}(0)=X(0), X^{''}(1)=X^{'}(1)\}.
\end{eqnarray*}

\begin{proposition} $L^3_M$ is a split Lagrangian subspace intersecting neatly $C_M\oplus C_M \oplus C_M$.
\end{proposition}
\proof{
$L^3_M$ is isotropic since
\begin{eqnarray*}
&&\omega\oplus \omega \oplus \overline{\omega}((X_1,\Omega(\frac{dX_1}{dt}),X_1^{'}, \Omega(\frac{dX_1^{'}}{dt}),X_1^{''}, \Omega(\frac{dX_1^{''}}{dt}),(X_2,\Omega(\frac{dX_2}{dt}),X_2^{'}, \Omega(\frac{dX_2^{'}}{dt}),X_2^{''}, \Omega(\frac{dX_2^{''}}{dt}))\\
&&X_1(1)\Omega X_2(1)-X_1(0)\Omega X_2(0)+X^{'}_1(1)\Omega X_2^{'}(1)-X_1^{'}(0)\Omega X_2^{'}(0)-X^{''}_1(1)\Omega X_2^{''}(1)-X_1^{''}(0)\Omega X_2^{''}(0)\\
&&=0.
\end{eqnarray*}
Similarly as the previous case, the following is an isotropic complement $(L^3_M)^C$ of $L^3_M$ in $V\oplus V \oplus \overline V$:
$$
(L^3_M)^C:=\left\{ \begin{array}{lr}
(\mbox { ct },\eta,\mbox{ ct }^{'}, \eta^{'}, \mbox{ct }^{''}, \eta^{''})\in V\oplus V \oplus \overline V\colon &\\
  \int_0^1\eta(t)\,dt=\int_0^1 \eta^{'}(t)\, dt=\int_0^1\eta^{''}(t)\, dt=0\\
       \end{array} \right\}.
$$
\qed}

\subsection{The constant case}
The case where $M=\mathbb R^n$ and $\Pi$ is a constant Poisson bivector is a combination of the previous two cases. More precisely, if $\Pi$ is constant, we can decompose $\Pi$ as a direct sum 
$\Omega\oplus 0$ ,$M$ as a direct sum $M_\Omega\oplus M_0$ and $V$ as $V_{\Omega} \oplus V_{0}$. This decomposition induces a decomposition of the Euler-Lagrange space
$$C_M=C_{\Omega} \oplus C_{0}.$$
It is easy to check that the direct sum of split coisotropic subspaces and split Lagrangians subspaces  is again split, therefore the following proposition holds:
\begin{proposition}\label{constant} If $\Pi$ is a constant Poisson bivector on $M$, then
$C_M$ is a split coisotropic subspace of $V$, and similarly the evolutions relations $L_1, L_2$ and $L_3$  intersecting neatly $C_M$, $C_M\oplus C_M$ and $C_M\oplus C_M \oplus C_M$ respectively.
\end{proposition}
\subsection{The general case}
We will prove that the above considerations for $C_M$ and the relations $L^1_M, L^2_M, L^3_M$ can be naturally extended for any Poisson bivector field $\Pi$ on $M=\mathbb R^n$. In \cite{Dual}, the linearized version of the Euler-Lagrange equations is modified via  a change of coordinates, and the resulting Euler-Lagrange space is identified with the solutions $(\lambda, \varphi)$ for the equation
$\frac{d\lambda}{dt}= -\tilde P^{\sharp} {\varphi}$
where $\tilde P^{\sharp} $ is constant. The change of coordinates leaves invariant the weak symplectic structure of $V$ and hence the splitting condition for  coisotropic and Lagrangian subspaces is preserved. \\
In order to show this, let us briefly recall the change of coordinates from \cite{Dual}. If we fix an element $(X,\eta)$ in $C_M$ and if we choose a linear connection $\nabla$ on $M$, it induces a connection $\nabla^{*}$ in $X^{*}TM$, described in local coordinates by the following exterior covariant derivative $\partial$:
\begin{eqnarray*}
\partial\colon \Gamma (X^*TM)&\to& \Gamma (T^*[0,1] \otimes X^{*}TM)\colon\\
\sigma^i &\mapsto& \partial \sigma^i:= d\sigma^i + \Gamma^i_{rs} dx^r\sigma ^s,
\end{eqnarray*}
where $\Gamma ^i_{rs}$ are the Christoffel symbols for $\nabla$.
We can now deform this connection using the Poisson bivector $\Pi$, namely, if we denote 
$$A:= \nabla^{*} \Pi^{\sharp}(X) \eta $$ described locally by
$$(\nabla^{*} \Pi)^{ij}_k=\partial_k \Pi^{ij}+ \Gamma^i_{kr}\Pi^{rj}+ \Gamma^j_{kr}\Pi^{ir},$$
we can construct the covariant derivative 
$$D:= \partial +A.$$
Proposition 2.1 in \cite{Dual} shows that the space $T_{(X,\eta)}C_M$ can be written as follows:
\begin{eqnarray}
T_{(X,\eta)}C_M&=&\{(\xi, e) \in \Gamma (X^*TM)\oplus \Gamma(T^{*}[0,1]\otimes X^{*}T^{*}M)\colon\\
 && D\xi=-\Pi^{\sharp}(X)e \label{Euler}\}.
\end{eqnarray}
Furthermore, if $U$ denotes the parallel transport of $D$, the following change of coordinates:
\begin{eqnarray}
\xi&\mapsto& \lambda:= U \xi \in \Omega^0 ([0,1], T_{X(0)}M)\\
e &\mapsto& \varphi:=(U^{t})^{-1}e \in \Omega^1([0,1], T^{*}_{X(0)}M)\\
\Pi^{\sharp}&\mapsto& P^{\sharp}:= U\Pi^{\sharp}U^t \in \Omega^0([0,1], \mbox{ Hom }(T^{*}_{X(0)}M, T_{X(0)}M)),
\end{eqnarray}
maps the equation \ref{Euler} to 
\begin{equation}\frac{d\lambda}{dt}= -P^{\sharp} \varphi. \label{New}
\end{equation}
Lemma 2.2 in \cite{Dual} states that  if $\Pi$ is a Poisson bivector, then $P$ is skew symmetric and constant.
This implies that the space $C_M(\lambda, \varphi)$ of solutions of Equation \ref{New} can be identified with the Euler Lagrange space $C_M \subset V(\lambda, \varphi)$ for the constant Poisson structure $P^{\sharp}$, thus it  is split coisotropic by Proposition \ref{constant}. 
By using the fact that the space $V(\lambda, \varphi)$ is weak symplectomorphic to $V(X,\eta)$ using the changes of coordinates stated above, the following result holds:
\begin{theorem}
Let $M=\mathbb R^n$ and $\Pi$ be a Poisson bivector on $M$. Then the linearized Euler-Lagrange space $T_{(X,\eta)}C_M$ is a split coisotropic subspace of  $V(X,\eta)$, for all $(X,\eta) \in C_M$.
\end{theorem}
In \cite{Relational} it is shown that the evolution relations $L^1_M, L^2_M$ and $L^3_M$ can be written only in terms of $C_M$ and its symplectic orthogonal $C_M^{\omega}$. A similar argument as before proves that
\begin{theorem}\label{StrongRel}
The tangent spaces of the evolution relations $T L^1_M, TL^2_M$ and $T L^3_M$ described by
\begin{eqnarray*}
TL^1_M&=&\{(\delta X, \delta \eta) \in TC_M \colon  \exists \tilde X= \mbox {ct }, \tilde {\eta} \in \ker{\Pi^{\sharp}}: (\delta X -\delta \tilde X, \delta \eta- \delta \tilde{\eta})\in (TC_M)^{\omega})\}\\
TL^2_M&=& \{((\delta X, \delta \eta), (\delta \tilde X, \delta \tilde{\eta})) \in TC_M\oplus TC_M: (\delta X-\delta \tilde X, \delta \eta- \delta \tilde{\eta})\in (TC_M)^\omega\}\\
TL^3_M&=&\{((\delta X, \delta \eta),(\delta \tilde X, \delta \tilde{\eta} ),(\delta\stackrel{\approx}{X}, \delta \stackrel{\approx}{\eta})) \in TC_M\oplus TC_M\oplus TC_M\\
&:& ((\delta X, \delta \eta)*(\delta \tilde X, \delta \tilde{\eta} ),(\delta\stackrel{\approx}{X}, \delta \stackrel{\approx}{\eta}) )\in TL^2_M\},
\end{eqnarray*}
 are split Lagrangian subspaces and they intersect neatly  $TC_M, TC_M\oplus TC_M$ and $TC_M \oplus TC_M \oplus TC_M$ respectively, for any Poisson bivector $\Pi$.
\end{theorem}
\begin{remark}
A possible further applications in field theory of the formalism of split Lagrangian and coisotropic spaces  is to consider the cases in which the symplectic structure is strong. In order to achieve that, one has to restrict to the case in which the spaces of fields are Sobolev spaces, and we can describe them as reflexive spaces. A natural example is quantum mechanics and  scalar field theory on Riemannian manifolds, in which the space of states is chosen to be $l_2$ functions, rather than $C^{\infty}$-functions.  This viewpoint imposes some conditions over the solutions of the Euler-Lagrange equations, possibly allowing weak solutions to the $EL(M)$-spaces. In several cases, the spaces of weak and strong solutions coincide, but there are several analytical issues to be consider while undertsanding the symplectic formulation of weak solutions of PDE.
\end{remark}
\remark{(Split canonical relations and symplectic groupoids).
In \cite{Relational, Con}, the evolution relations for the Poisson sigma model are used to construct an integration of the Poisson manifold $M$. These relations are the analogues of the graphs of the unit, identity and multiplication maps of a conventional symplectic groupoid $G\rightrightarrows M $ integrating the Lie algebroid $T^*M$. Theorem  \ref{StrongRel} shows that the relational symplectic groupoid associated to any Poisson manifold $M$ via the Poisson sigma model  is split, in the sense that its associated canonical relations $L^1, L^2$ and $L^3$ are split.}

\begin{remark}By  using smooth sections of the cotangent bundle of the path space, we consider the Poisson sigma model in terms of weak symplectic Fr\'echet spaces. However, in the path Sobolev space $L_2^1([0,1], \mathbb R)$ of $L_2$-differentiable paths, the canonical symplectic form on the cotangent bundle is strong symplectic. The price to pay here is that the Euler-Lagrange space $EL(M)$ must be redefined to allow weak solutions of the O.D.E.  given by Equation \ref{EL}.
\end{remark}


\end{document}